\documentclass{amsart}

\usepackage{amssymb}
\usepackage{amsmath}
\usepackage[unicode,linkcolor=blue,hyperindex]{hyperref}
\usepackage{url,breakurl}
\usepackage{multirow}
\usepackage{multicol}
\usepackage{graphicx}
\usepackage{enumerate}

\newcommand{\rd}{{\mathbb{R}^d}}

\newcommand{\R}{{\mathbb {R}}}

\newcommand{\V}{{\mathbb {V}}}
\newcommand{\poly}{{\mathbb {P}}}
\newcommand{\T}{{\mathcal {T}}}
\newcommand{\M}{{\mathcal {M}}}
\renewcommand{\P}{{\mathcal {P}}}
\newcommand{\tHs}{{\widetilde H^{s}(\Omega)}}
\newcommand{\Wztp}{ \mbox{ \raisebox{7.2pt} {\tiny$\circ$} \kern-10.7pt} {W}^t_p }
\newcommand{\Wzop}{ \mbox{ \raisebox{7.2pt} {\tiny$\circ$} \kern-10.7pt} {W}^1_p }
\newcommand{\Wzrp}{ \mbox{ \raisebox{7.2pt} {\tiny$\circ$} \kern-10.7pt} {W}^r_p }

\newcommand{\SZ}{{\Pi_\T}}
\newcommand{\sob}{\mbox{Sob\,}}

\newcommand{\REFINE}{\textsf{REFINE}}
\newcommand{\GREEDY}{\textsf{GREEDY}}

\newcommand{\wt}{\widetilde}

\newcommand{\pp}{\partial}
\newcommand{\eps}{\varepsilon}

\newcommand{\Om}{\Omega}

\renewcommand{\l}{t}

\newcommand{\dist}{{\mbox{dist}}}

\newtheorem{theorem}{Theorem}[section]
\newtheorem*{theorem*}{Theorem}
\newtheorem{lemma}[theorem]{Lemma}
\newtheorem{proposition}[theorem]{Proposition}
\newtheorem{corollary}[theorem]{Corollary}

\theoremstyle{remark}
\newtheorem{remark}[theorem]{Remark}

\theoremstyle{definition}

\numberwithin{equation}{section}
\numberwithin{table}{section}
\numberwithin{figure}{section}

\newenvironment{algotab}%
{\par\begin{samepage}%
\begin{tabbing}\ttfamily%
 \hspace*{5mm}\=\hspace{3ex}\=\hspace{3ex}\=\hspace{3ex}\=\hspace{3ex}%
\=\hspace{3ex}\=\hspace{3ex}\=\hspace{3ex}\=\hspace{3ex}\kill}%
{\end{tabbing}\end{samepage}}

\title[Constructive approximation for the integral fractional Laplacian]{Constructive approximation on graded meshes for the integral fractional Laplacian}

\author[J.P.~Borthagaray]{Juan Pablo~Borthagaray}
\address[J.P.~Borthagaray]{Departamento de Matem\'atica y
Estad\'istica del Litoral, Universidad de la Rep\'ublica, Salto,
Uruguay. Current address: Centro de Matem\'atica, Universidad de la Rep\'ublica, Montevideo, Uruguay}
\email{jpb@cmat.edu.uy}
\thanks{JPB has been supported in part by Fondo Vaz Ferreira grant 2019-068}

\author[R.H.~Nochetto]{Ricardo H.~Nochetto}
\address[R.H.~Nochetto]{Department of Mathematics and Institute for
Physical Science and Technology, University of Maryland, College
Park, MD 20742, USA}
\email{rhn@math.umd.edu}
\thanks{RHN has been supported in part by NSF grant DMS-1411808}

\begin{document}

\begin{abstract}
We consider the homogeneous Dirichlet problem for the integral fractional Laplacian $(-\Delta)^s$. We prove optimal Sobolev regularity estimates in Lipschitz domains provided the solution is $C^s$ up to the boundary. We present the construction of graded bisection meshes by a greedy algorithm and derive quasi-optimal convergence rates for approximations to the solution of such a problem by continuous piecewise linear functions. The nonlinear Sobolev scale dictates the relation between regularity and approximability.
\end{abstract}

\maketitle

\section{introduction}

We consider the integral fractional Laplacian on a bounded domain $\Omega \subset \rd$,
\begin{equation} \label{eq:def-Laps}
(-\Delta)^s u (x) = \Lambda(d,s) \mbox{ p.v.} \int_\rd \frac{u(x)-u(y)}{|x-y|^{d+2s}}, \qquad \Lambda(d,s) = \frac{2^{2s} s \Gamma(s+\frac{d}{2})}{\pi^{d/2} \Gamma(1-s)} .
\end{equation}
and corresponding homogeneous Dirichlet problem \cite{BBNOS18, DElia, Lischke}
\begin{equation} \label{eq:Dirichlet}
\left\lbrace
  \begin{array}{rl}
      (-\Delta)^s u = f & \mbox{ in }\Omega, \\
      u = 0 & \mbox{ in }\Omega^c = \R^d\setminus \overline \Omega, \\
      \end{array}
    \right.
\end{equation}
where $0<s<1$. 
We assume that $\Omega \subset \rd$  is a bounded Lipschitz domain, $u$ is $C^s$ up to the boundary, and $f \colon \Omega \to \R$ a bounded function. We are interested in analyzing the performance of a {\it greedy} algorithm for approximating solutions to \eqref{eq:Dirichlet} by continuous piecewise linear functions over graded bisection meshes.

Regardless of the regularity of the domain $\Omega$ and the right-hand side $f$, solutions to \eqref{eq:Dirichlet} generically develop an algebraic singular layer of the form (cf. e.g. \cite{Grubb})
\begin{equation} \label{eq:bdry_behavior}
u(x) \approx \dist(x,\pp\Omega)^s,
\end{equation}
that limits the smoothness of solutions. Heuristically, let us consider that $\Omega \subset \R$ is the half-line $\Omega = (0, \infty)$. Thus, we can interpret the behavior \eqref{eq:bdry_behavior} as $u(x) \approx x_+^s$, and wonder under what conditions this function belongs to a Sobolev space with differentiability index $r$ and integrability index $p$. For that purpose, let us compute (Riemann-Liouville) derivatives of order $r>0$ of the function $v : \R \to \R$, $v(x) =  x_+^s$:
\[
\partial^r v (x) \simeq x_+^{s-r}, \quad r > 0.
\]
We observe that $\partial^r v$ is $L^p$-integrable near $x=0$ if and only if $p(s-r) > - 1$, namely, if $r<s+\frac1p$. This heuristic discussion illustrates the natural interplay $r<s+\frac1p$ between the differentiability order $r$ and integrability index $p$ for membership of solutions to \eqref{eq:Dirichlet} in the class $W^r_p$, at least for dimension $d=1$. It turns out that the restriction $r<s+\frac1p$ is needed irrespective of dimension (cf Theorem \ref{thm:higher-regularity} below).

For the sake of approximation, one can find the optimal choice of the indexes $r,p$ by inspecting a DeVore diagram; see Figure \ref{fig:DeVore} for an illustration in the two-dimensional setting. Recall the definition of Sobolev number $\sob (W^r_p):=r-\frac{d}{p}$ and the Sobolev line corresponding to the nonlinear approximation scale of $H^s$,
\[
\left\{ \left(\frac{1}{p}, r\right) \colon \sob(W^r_p) = \sob (H^s) \right\} = \left\{ \left(\frac{1}{p}, r\right) \colon r = s - 1 + \frac{2}{p} \right\}.
\]
In order to have a compact embedding $W^r_p \subset H^s$, we require $\sob(W^r_p) > \sob (H^s)$ or equivalently $(1/p,r)$ to lie above this line, as well as $r>s$.
In addition, the regularity restriction $r<s+\frac1p$, derived heuristically earlier, is depicted in red and intersects the Sobolev line at $(1,1+s)$.
\begin{figure}[htbp]
\begin{center}
  \includegraphics[width=0.65\linewidth]{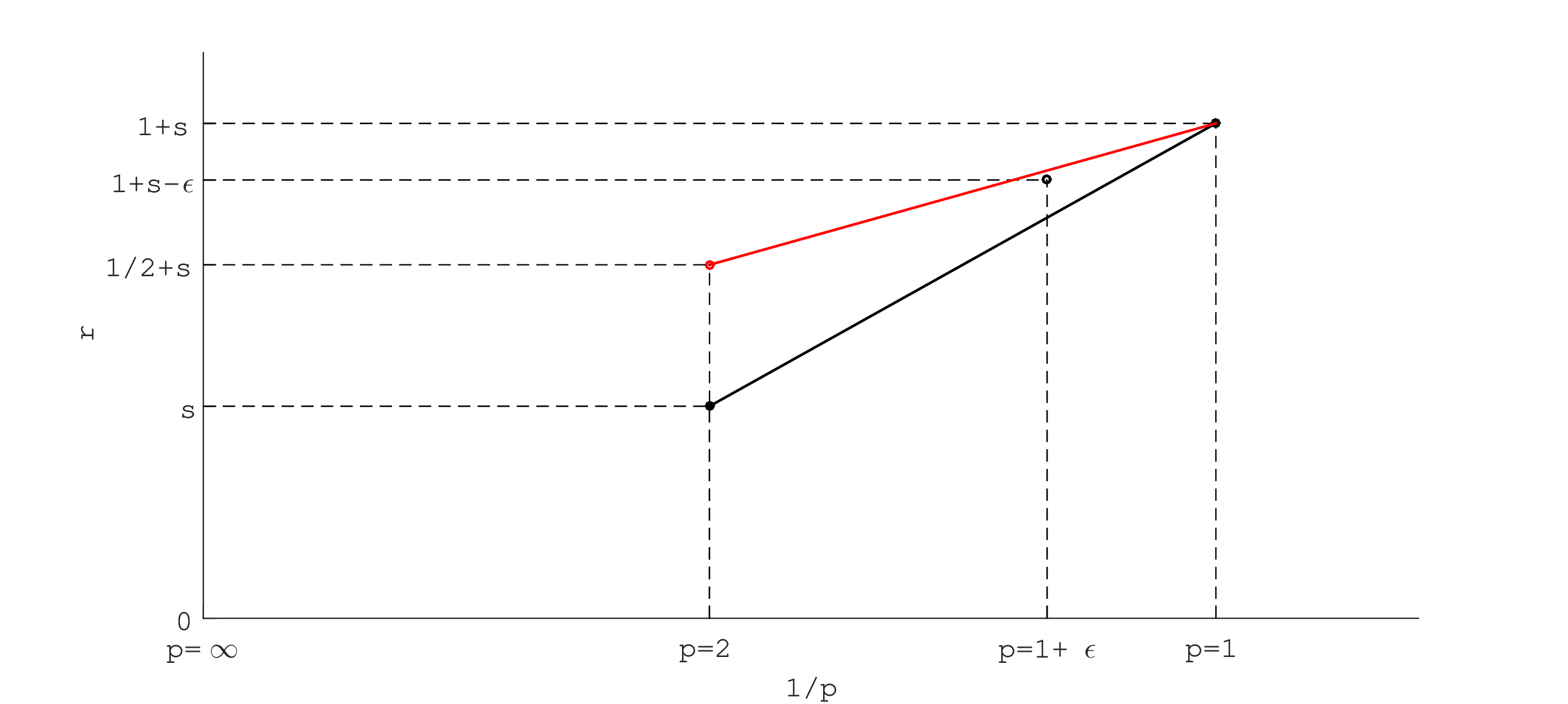}
\end{center}
\vspace{-0.5cm}
\caption{DeVore diagram for dimension $d = 2$. Our heuristic argument indicates that solutions to \eqref{eq:Dirichlet} are of class $W^r_p$, where $(1/p,r)$ must be below the regularity (red) line $ \{ (1/p, r) \colon r = s + \frac{1}{p} \}$. Additionally, we depict the Sobolev line (black) $\{ (1/p, r) \colon r = s - 1 + \frac{2}{p} \}$ that connects the spaces $W^r_p$ and $H^s$ and corresponds to the nonlinear approximation scale.}
\label{fig:DeVore}
\end{figure}
Letting $p=1+\eps$ and $r=s+1-\eps$ with $\eps > 0$ arbitrarily small, we have
\[
\sob(W^{s+1-\eps}_{1+\eps}) =  s - 1 + \eps  \left(\frac{1-\eps}{1+\eps} \right) > \sob(H^s),
\]
while the condition $\frac{1}{p}=\frac{1}{1+\eps} > 1 - \eps=r-s$ is also satisfied. This yields an optimal choice of parameters in dimension $d=2$.

One can perform an analogous argument for arbitrary dimension $d$: the optimal approximation space can be found on a DeVore diagram by intersecting the Sobolev line corresponding to $H^s$ with the regularity line $\{r = s + \frac1p \}$. For $d=1$ these two lines are parallel, while for $d \ge 2$ these lines meet at the point $(1/p, r)$ with
\[
p = \frac{2(d-1)}{d},
\quad
r = s + \frac{d}{2(d-1)}.
\]
This indicates that the optimal regularity one can expect corresponds to the differentiability order $r = s + \frac{d}{2(d-1)} < s+1$ for $d>2$.

In this paper we justify this heuristic argument rigorously and exploit it to construct suitable graded bisection meshes via a greedy algorithm that delivers quasi-optimal convergence rates for continuous piecewise linear approximation.
In Section \ref{sec:notation}, we introduce some notation regarding fractional-order Sobolev spaces and the weak formulation of \eqref{eq:Dirichlet}. Section \ref{sec:regularity} is devoted to providing a rigorous proof of the regularity estimates discussed above. In Section \ref{S:graded-meshes} we study the performance of the greedy algorithm. Finally, Section \ref{S:numerical-experiments} includes some numerical experiments for $d=2$ illustrating our theoretical findings: we observe optimal convergence rates and that the singular boundary layer \eqref{eq:bdry_behavior} dominates reentrant corner singularities.

\section{Fractional Sobolev Spaces}  \label{sec:notation}

In this section we set the notation and review some properties of the spaces involved in the rest of the paper. We start by recalling some function spaces. 

Given $\sigma \in (0,1)$ and $p \in (1,\infty)$, we consider the seminorm
\begin{equation}\label{eq:def-Sobolev}
|v|_{W^\sigma_p(\Omega)} := \left( \Lambda(\sigma,d,p) \iint_{\Omega\times\Omega} \frac{|v(x)-v(y)|^p}{|x-y|^{d+\sigma p}} dx \, dy \right)^{\frac{1}{p}}.
\end{equation}
Above, we set the constant $\Lambda(\sigma,d,p)$ in such a way that in the limits $\sigma \to 0$ and $\sigma \to 1$ one recovers the standard integer-order norms. More precisely, by the results in \cite{BourgainBrezisMironescu} and \cite{MaSh02}, we require
\begin{equation} \label{eq:asymptotics-C}
 \Lambda(\sigma,d,p) \simeq  \frac{\sigma p \Gamma\left(\frac{d}2\right)}{4 \pi^{d/2}} \mbox{ for } \sigma \simeq 0, \quad \Lambda(\sigma,d,p) \simeq  \frac{(1-\sigma) p}{\int_{S^{d-1}} |\omega \cdot e_1|^p  \, d\omega} \mbox{ for } \sigma \simeq 1.
\end{equation}
{We see that these constants vanish linearly in $\sigma$ as $\sigma\to0$ and $1-\sigma$ as $\sigma\to1$.
For $p=2$ we set the constant $\Lambda(\sigma,d,2) = \frac{\Lambda(d,\sigma)}{2}$ as in the definition \eqref{eq:def-Laps} of the fractional Laplacian $(-\Delta)^s$, which is consistent with these requirements. 

We adopt the convention that zero-order derivatives correspond to the identity, and write $W^0_p(\Omega) = L^p(\Omega)$, and $D^0 v = v$. Given $t \in (0,2)$, let $k = \lfloor t \rfloor$ be the largest integer number smaller or equal than $t$, $\sigma = t - k \in [0,1)$, and we define
\[
W^t_p (\Omega) := \left\{ v \in W^{k,p}(\Omega) \colon |D^\alpha v | \in W^\sigma_p(\Omega) \ \forall \alpha \text{ with } |\alpha| = k \right\} ,
\]
with the norm
\[
\| v \|_{W^t_p (\Omega)} := \| v \|_{W^k_p(\Omega)} + |v|_{W^t_p (\Omega)}, \qquad |v|_{W^t_p (\Omega)} = \sum_{|\alpha| = k } | D^\alpha v |_{W^\sigma_p(\Omega)}.
\]
The {\it Sobolev number} of $W^t_p (\Omega)$ is defined to be $\sob (W^t_p) := t - d/p$.

For our purposes, we need to consider zero-extension spaces as well. For $v \colon \Omega \to \R$, we denote by $\wt{v}$ its extension by zero on $\Omega^c$. If $t \in (0,2)$ and $k = \lfloor t \rfloor$, we define $\widetilde{W}^t_p(\Omega)$ to be the space of functions whose trivial extensions are globally in $W^t_p(\rd)$,
\begin{equation*} \label{eq:def-tildeWsp}
\widetilde{W}^t_p(\Omega) = \{ v \in W^k_p(\Omega) \colon \wt{v} \in W^t_p(\rd) \}.
\end{equation*}
These spaces characterize the regularity of functions across $\partial\Omega$. It is important to realize that if $p \in (1, \infty)$, $t- \lfloor t \rfloor \in (0,1/p)$, then the spaces $W^t_p(\Omega)$ and $\widetilde{W}^t_p(\Omega)$ are identical; equivalently, any function in $W^t_p(\Omega)$ can be extended by zero without changing its regularity (cf. \cite[Corollary 1.4.4.5]{Grisvard}). In contrast, if $t- \lfloor t \rfloor \in (1/p,\infty)$, then the notion of trace is well defined in $W^t_p(\Omega)$ and its subspace $\Wztp(\Omega)$ of functions with vanishing trace coincides with $\widetilde{W}^t_p(\Omega)$. Finally, the case $t- \lfloor t \rfloor = 1/p$ is exceptional and corresponds to the so-called Lions-Magenes space \cite[Theorem 1.11.7]{LionsMagenes}.

From now on, for any given function $v \in \widetilde W^t_p(\Omega)$ we will drop the tilde to denote its zero extension, and assume that the domain of $v$ is $\rd$. An important feature of these zero-extension spaces is the following Poincar\'e inequality: if $v \in \widetilde{W}^s_p(\Omega)$ for $p \in (1,\infty)$ and $t\in(0,2)$, then $\| v \|_{W^{\lfloor t \rfloor}_p(\Omega)} \le C |v|_{W^t_p(\rd)}$. Therefore, 
\[
\| v \|_{\widetilde W^t_p(\Omega)} := |v|_{W^t_p(\rd)}
\]
defines a norm equivalent to $\| \cdot \|_{W^t_p(\rd)}$ in $\widetilde{W}^t_p(\Omega)$.

As usual, we denote Sobolev spaces with integrability index $p=2$ by using the letter $H$ instead of $W$. Hence $\wt{H}^s(\Omega)=\wt{W}^s_2(\Omega)$, and we define $H^{-s}(\Omega) = [\tHs]'$ and $\langle \cdot , \cdot \rangle$ its duality pairing. For $f \in H^{-s}(\Omega)$, the weak formulation of \eqref{eq:Dirichlet} reads: find $u \in \tHs$ such that
\begin{equation}\label{eq:weak}
a(u,v) := \frac{\Lambda(d,s)}{2} \iint_{\rd \times \rd} \frac{(u(x)-u(y))(v(x)-v(y))}{|x-y|^{d+2s}} dx \, dy = 
\langle f , v \rangle
\end{equation}
for all $v \in \tHs$. Existence and uniqueness of weak solutions in $\tHs$, as well as the stability of the solution map $f \mapsto u$, are straightforward consequences of the Lax-Milgram theorem. We point out that, in the left-hand side of \eqref{eq:weak}, the integration region is effectively $(\Omega \times \rd) \cup (\rd \times \Omega)$.

\section{Regularity of solutions} \label{sec:regularity}

The purpose of this section is to provide regularity estimates for solutions of \eqref{eq:Dirichlet} in terms of fractional Sobolev norms with arbitrary integrability index $p \in (1,\infty)$. In a similar fashion to \cite{AcosBort2017fractional}, our starting point shall be the precise weighted H\"older estimates derived by X. Ros-Oton and J. Serra \cite{RosOtonSerra}. In \cite{AcosBort2017fractional} these estimates were employed to obtain regularity estimates in weighted Sobolev spaces with differentiability $1+s$, where the weight is a power of the distance to the boundary of $\Om$. As we show below, such regularity estimates are optimally suited for the case $d=2$. Here, we derive optimal regularity estimates for any $d \ge 1$. Our technique consists in recasting the estimates from \cite{RosOtonSerra} in unweighted Sobolev spaces with differentiability index $s+\frac1p$ but at the expense of an integrability index $1< p < 2$.

\subsection{H\"older regularity.}
We start with two important assumptions. We first assume that the solution $u \in C^s(\overline{\Omega})$ , or equivalently, exploiting the zero extension
\begin{equation} \label{eq:Holder-regularity}
\| u \|_{C^s(\rd)} \le C \| f \|_{L^\infty(\Omega)} .
\end{equation}
We recall that \eqref{eq:Holder-regularity} is proved in \cite[Proposition 1.1]{RosOtonSerra} provided $\Omega$ satisfies an {\it exterior ball condition}. It is also a consequence of \cite[Theorem 1.4]{AbatangeloRosOton:2020} provided $\Omega$ is of class $C^{1,\beta}$ and $f\in C^{\beta-s}(\overline{\Omega})$ for $s<\beta<1$.
Secondly, we shall assume $f$ possesses certain H\"older continuity. Combining these two assumptions allows us to derive higher-order regularity estimates on the solution.

We will employ the letter $\delta$ to denote either the distance functions
\[ \begin{aligned}
& \delta (x) = \text{dist}(x,\partial \Omega),  & x \in \Omega, \\
& \delta (x,y) = \min \{ \delta(x), \delta(y) \}, & x,y \in \Omega.
\end{aligned} \]
For $\beta>0$, we denote by $|\cdot|_{C^\beta(\Omega)}$ the $C^\beta(\Omega)$-seminorm. If $\theta \geq - \beta$, let us set $\beta = k + \beta'$ with $k$ integer and $\beta' \in (0,1]$. We consider the seminorm 
\[
|v|_{\beta}^{(\theta)} = \sup_{x,y \in \Omega} \delta(x,y)^{\beta+\theta} \frac{|D^k v (x) - D^k v(y)|}{|x-y|^{\beta'}},
\]
and the associated norm $\| \cdot \|_{\beta}^{(\theta)}$ in the following way: for $\theta \geq 0$,
\[
\| v \|_\beta^{(\theta)} = \sum_{\ell=0}^k \left( \sup_{x \in \Omega} \delta(x)^{\ell+\theta} |D^\ell v(x)| \right)+ |v|_{\beta}^{(\theta)},
\]
while for $-\beta<\theta<0$,
\[
\| v \|_\beta^{(\theta)} = \| v \|_{C^{-\theta}(\Omega)} + \sum_{\ell=1}^k \left( \sup_{x \in \Omega} \delta(x)^{\ell+\theta} |D^\ell v (x)| \right) + |v|_\beta^{(\theta)} .
\]

The following estimate \cite[Proposition 1.4]{RosOtonSerra} is essential in what follows. It hinges on \eqref{eq:Holder-regularity} and H\"older continuity of $f$ rather than any specific regularity of $\partial\Omega$.
\begin{theorem}[weighted H\"older regularity]\label{thm:ROS2}
Let $\Omega$ be a bounded Lipschitz domain and $\beta>0$ be such that neither $\beta$ nor $\beta+2s$ is an integer. Let $f \in C^{\beta}(\Omega)$ be such that $\|f\|_\beta^{(s)} < \infty$, and $u \in C^s(\rd)$ be the solution of \eqref{eq:weak}. Then, $u\in C^{\beta+2s}(\Omega)$ and 
\[
\| u \|_{\beta + 2s }^{(-s)} \leq C(\Omega,s,\beta) \left( \|u\|_{C^s(\rd)} + \| f \|_{\beta}^{(s)} \right) .
\]
\end{theorem}

We next recast this estimate depending $\beta+2s$, with the exceptional case $s=\frac12, \beta>0$. According to \eqref{eq:Holder-regularity} and the definition of $\| f \|_{\beta}^{(s)}$, we have
\[
\|u\|_{C^s(\rd)} + \| f \|_{\beta}^{(s)}  \le C \|f\|_{C^\beta(\overline{\Omega})}.
\]

\begin{corollary}[pointwise weighted bounds]\label{C:pointwise-bounds}
Let $\Omega$ be a bounded Lipschitz domain, the solution $u$ of \eqref{eq:weak} satisfy \eqref{eq:Holder-regularity}, and $\beta>0$. If $f \in C^{\beta}(\overline\Omega)$, then
\begin{enumerate}[$\bullet$]
\item
{\it Case $\beta+2s <1 $}: we have $s < \frac12$ and
\begin{equation} \label{eq:weighted-beta}
\sup_{x, y \in \Omega} \delta(x,y)^{\beta+s} \, \frac{|u(x)-u(y)|}{|x-y|^{\beta+2s}} \leq C \| f \|_{C^\beta(\overline\Omega)}.
\end{equation}

\item
{\it Case $s=1/2$ and $\beta>0$}: we have
\begin{equation} \label{eq:weighted-beta1/2}
 \sup_{x \in \Omega} \delta(x)^{1/2} |D u (x)| \le C \| f \|_{C^\beta(\overline\Omega)}.
\end{equation}

\item
{\it Case $1 < \beta+2s < 2$}: we have
\begin{equation} \label{eq:higher-beta}
\sup_{x \in \Omega} \delta(x)^{1-s} |D u (x)| + 
\sup_{x,y \in \Omega} \delta(x,y)^{\beta+s} \frac{| D u (x) - D u (y)|}{|x-y|^{\beta+2s-1}}
\leq C \| f \|_{C^{\beta}(\overline \Omega)}. 
\end{equation}
\end{enumerate}
\end{corollary}

The following interior H\"older estimate \cite[Lemma 2.9]{RosOtonSerra} will also be used later.

\begin{lemma}[interior regularity]
If $f \in L^\infty(\Omega)$ and $\gamma \in (0,2s)$, then $u$ verifies
\begin{equation}
|u|_{C^\gamma (\overline{B_R (x)})} \leq C R^{s-\gamma} \| f \|_{L^\infty(\Omega)} \ \forall x \in \Omega,
\label{eq:local-Holder} \end{equation}
where R = $\frac{\delta(x)}{2}$ and the constant C depends only on $\Omega, s$ and $\gamma$, and blows up only when $\gamma \to 2s$.
\end{lemma}

\subsection{Sobolev regularity.}
Our goal for the remainder of this section is to use the H\"older estimates we have reviewed to derive bounds on Sobolev norms of $u$. We first show that under suitable assumptions on the right-hand side $f$, the first-order derivatives of the solution $u$ are $L^p$-integrable. For such a purpose, we resort to the following result \cite{BourgainBrezisMironescu}, which utilizes the asymptotic behavior $\Lambda(\sigma,d,p) \simeq 1-\sigma$ as $\sigma \uparrow 1$ of the scaling factor $\Lambda(\sigma,d,p)$ in the definition \eqref{eq:def-Sobolev} of the seminorm $|\cdot|_{W^\sigma_p(\Omega)}$.

\begin{proposition}[limits of fractional seminorms] \label{prop:BBM}
Assume $v \in L^p (\Omega)$, $1<p<\infty$. Then, it holds that 
\begin{equation}\label{eq:BBM} 
\lim_{\eps \to 0} |v|_{W^{1-\eps}_p({\Omega})} = |v|_{W^1_p({\Omega})}.
\end{equation}
\end{proposition}

\begin{remark}[integrability of powers of the distance function to the boundary]
On Lipschitz domains, powers of the distance function to the boundary have the following integrability property: for every $\alpha < 1$ it holds that (cf. for example \cite[Lemma 2.14]{Mazya})
\begin{equation}\label{eq:Mazya}
\int_\Omega \delta(x)^{-\alpha} dx = \mathcal{O}\left( \frac{1}{1-\alpha} \right).
\end{equation}
\end{remark}

\begin{theorem}[$W^1_p$-regularity] \label{thm:W1p-regularity}
Let $\Omega$ be a bounded Lipschitz domain, the solution $u$ to \eqref{eq:weak} obey \eqref{eq:Holder-regularity}, $s \in (0,1)$, $p \in (1, \infty)$ be such that $1 - 1/p < s$, and $f$ satisfy the following regularity assumptions:
\[
\left\lbrace \begin{array}{rl}
f \in C^{1-2s}(\overline\Omega)  & \text{if } s < 1/2, \\
f \in C^\beta(\overline\Omega) & \text{if } s = 1/2 \text{ (for some $\beta > 0$)} , \\
f \in L^\infty(\Omega) & \text{if } s > 1/2.
\end{array}\right.
\]
Then, $u$ satisfies $u \in \widetilde{W}^1_p(\Omega)$, which coincides with $\Wzop(\Omega)$.
\end{theorem}
\begin{proof}
We shall prove that, for every $\eps$ sufficiently small, $u \in W^{1-\eps}_p(\R^d)$ and
\begin{equation} \label{eq:BBM-Lambda}
|u|_{W^{1-\eps}_p(\R^d)}^p \le C(\Omega,d,s,p,f).
\end{equation}
From \eqref{eq:BBM} and the fact that $u$ is compactly supported, \eqref{eq:BBM-Lambda} implies that $u \in W^1_p(\R^d)$. Since $u = 0$ on $\Omega^c$ it follows that $u \in \Wzop(\Omega)=\widetilde{W}^1_p(\Omega)$ and $u$ has a well defined and vanishing trace because $p>1$. To exploit symmetry of the integrand in the definition of $|u|_ {W^{1-\eps}_p(\R^d)}^p$ we decompose the domain of integration $\Omega\times\Omega$ into
\[
D := \big\{(x,y)\in\Omega\times\Omega \colon \delta(x,y) = \delta(x) \big\}
= \big\{(x,y)\in\Omega\times\Omega \colon \delta(x) \le \delta(y) \big\}
\]
and its complement within $\Omega\times\Omega $ and realize that
\[
\iint_{\Omega\times\Omega} \frac{|u(x)-u(y)|^p}{|x-y|^{d+(1-\eps)p}} \, dy dx = 2 \iint_D \frac{|u(x)-u(y)|^p}{|x-y|^{d+(1-\eps)p}} \, dy dx.
\]
Similarly, for the rest of the domain of integration $[\Omega\times\Omega^c]\cup[\Omega^c\times\Omega]$ we have
\[
\iint_{[\Omega\times\Omega^c]\cup[\Omega^c\times\Omega]} \frac{|u(x)-u(y)|^p}{|x-y|^{d+(1-\eps)p}} \, dy dx
= 2 \iint_{\Omega\times\Omega^c} \frac{|u(x)-u(y)|^p}{|x-y|^{d+(1-\eps)p}} \, dy dx
\]
We further split the effective domain of integration into the sets
\[
A := \Big\{ (x,y) \in D \colon |x-y| < \frac{\delta(x)}{2} \Big\},
\quad
B := \big[ D \setminus A \big] \cup \big[ \Omega \times \Omega^c \big],
\]
and rewrite the seminorm $|u|_ {W^{1-\eps}_p(\R^d)}$ defined in \eqref{eq:def-Sobolev} as
\[
\frac{|u|_ {W^{1-\eps}_p(\R^d)}^p}{\Lambda(1-\eps,d, p)} = 2 \iint_{A} \frac{|u(x)-u(y)|^p}{|x-y|^{d+(1-\eps)p}} \, dy dx + 2 \iint_{B} \frac{|u(x)-u(y)|^p}{|x-y|^{d+(1-\eps)p}} \, dy dx ,
\]
where $\Lambda(1-\eps,d, p) \simeq \eps$ as $\eps\to0$ according to \eqref{eq:asymptotics-C}.
We finally fix $\eps<1-s$ if $s \ge 1/2$ and $\eps < 1 - 2s$ if $s < 1/2$, and estimate the
contributions on $A$ and $B$.

On the set $B$, we use the H\"older estimate \eqref{eq:Holder-regularity} and integration in polar coordinates together with \eqref{eq:Mazya} with $\alpha=p(1-s-\eps)<1$ to obtain
\[ \begin{aligned}
\iint_B \frac{|u(x)-u(y)|^p}{|x-y|^{d+(1-\eps)p}} \, dy dx & \le C \|f\|_{L^\infty(\Omega)}^p \int_{\Omega} \int_{ B(x,\delta(x)/2)^c} |x-y|^{-d-(1-\eps)p+sp} dy dx \\
& \le \frac{C}{p(1-\eps-s)} \|f\|_{L^\infty(\Omega)}^p \int_\Omega \delta(x)^{p(-1+\eps+s)} dx  \\
& \le \frac{C}{p(1-\eps-s)(1-p(1-\eps-s))} \|f\|_{L^\infty(\Omega)}^p.
\end{aligned} \]

To deal with the set A, we first assume that $s = 1/2$, note that $1-\frac{1}{p}<s$ yields $p < 2$, and employ \eqref{eq:weighted-beta1/2} together with \eqref{eq:Mazya} with $\alpha=p/2<1$ to get
\[
\int_\Omega |D u(x)|^p dx \le C \|f\|_{C^\beta(\overline{\Omega})} \int_\Omega \delta(x)^{-p/2} dx \le \frac{C}{2-p}
\|f\|_{C^\beta(\overline{\Omega})}.
\]
For $s \neq 1/2$ and distinguish two cases. In the case $s>1/2$,  we resort to \eqref{eq:local-Holder}
\[
\frac{|u(x)-u(y)|}{|x-y|^\gamma} \le C \delta(x)^{s-\gamma}  \|f\|_{L^\infty(\Omega)}
\quad \forall |x-y| \le \frac{\delta(x)}{2}
\]
and $\gamma \in (0,2s)$ to write
\begin{align*}
\iint_A \frac{\delta(x)^{p(s-\gamma)}}{|x-y|^{d+(1-\eps)p}} \, dy dx & 
\le C \int_{\Omega} \delta(x)^{p(s-\gamma)} \int_{B(x,\delta(x)/2)} |x-y|^{-d-(1-\eps-\gamma) p} dy dx \\
& \le \frac{C}{p(\gamma-1+\eps)} \int_\Omega \delta(x)^{p(-1+\eps+s)} dx,
\end{align*}
provided $\gamma \in (1-\eps, 2s)$. This implies
\[
\iint_A \frac{|u(x)-u(y)|^p}{|x-y|^{d+(1-\eps)p}} \, dy dx
\le \frac{C}{p(\gamma-1+\eps)(1-p(1-\eps-s))} \|f\|_{L^\infty(\Omega)}^p,
\]
and simply setting $\gamma = 1 - \eps/2$ yields
\[
\iint_A \frac{|u(x)-u(y)|^p}{|x-y|^{d+(1-\eps)p}} \, dy dx \le \frac{C}{\eps (1-p(1-\eps-s))}\|f\|_{L^\infty(\Omega)}^p.
\]

For $s < 1/2$, we resort to \eqref{eq:weighted-beta} with $\beta \in( 1 - 2s -\eps , 1- 2s)$, namely
\[
\frac{|u(x)-u(y)|}{|x-y|^{\beta+2s}} \le C \delta(x)^{-(\beta-s)} \|f\|_{C^{1 -2s}(\overline{\Omega})}
\]
because $f \in C^{1-2s}(\overline \Omega) \subset C^\beta(\overline\Omega)$. We next integrate in polar coordinates to get
\[ \begin{aligned}
\iint_A & \frac{|u(x)-u(y)|^p}{|x-y|^{d+(1-\eps)p}} \, dy dx \\
& \le C \|f\|_{C^{1-2s}(\overline{\Omega})}^p \int_\Omega \delta(x)^{-p(\beta+s)} \int_{B(x,\delta(x)/2)} |x-y|^{-d-(1-\eps-\beta-2s)p} dy dx \\
& \le \frac{C}{p(\beta+2s-1+\eps)} \|f\|_{C^{1-2s}(\overline{\Omega})}^p \int_\Omega \delta(x)^{p(s-1+\eps)} dx \\
& \le   \frac{C}{p(\beta+2s-1+\eps)(1-p(1-\eps-s))} \|f\|_{C^{1-2s}(\overline{\Omega})}^p.
\end{aligned} \]
Thus, letting $\beta = 1 -2s - \eps/2$ we obtain
\[
\iint_A \frac{|u(x)-u(y)|^p}{|x-y|^{d+(1-\eps)p}} \, dy dx \le  \frac{C}{\eps (1-p(1-\eps-s))} \|f\|_{C^{1-2s}(\overline{\Omega})}^p. 
\]

Collecting the estimates above and recalling the asymptotic behavior $\Lambda(1-\eps, d, p) \simeq \eps$ as $\eps\to0$ given in \eqref{eq:asymptotics-C}, we deduce the desired expression \eqref{eq:BBM-Lambda} for $s \neq 1/2$. This concludes the proof.
\end{proof}

We now aim to prove higher-order Sobolev regularity estimates.

\begin{theorem}[Sobolev regularity] \label{thm:higher-regularity}
Let $\Omega$ be a bounded Lipschitz domain, the solution $u$ to \eqref{eq:weak} satisfy \eqref{eq:Holder-regularity}, $s \in (0,1)$, and $f\in C^\beta(\overline \Omega)$ for some $\beta \in (0, 2-2s)$. Furthermore, let $r \in (s,\beta + 2s)$ and $r< s + \frac{1}{p}$. Then, $u \in \widetilde W^r_p(\Omega)$, with
\begin{equation}\label{eq:higher-regularity}
\|u\|_{\widetilde W^r_p(\Omega)}^p \leq \frac{C(\Omega, s,d,\beta)}{\big(\beta+2s-r\big)\big(1-p (r - s)\big)} \| f \|_{C^\beta(\overline \Omega)}^p,
\end{equation}
where the constant $C(\Omega, s,d,\beta)$ is robust with respect to $r\to1$ and $p\to\infty$.
\end{theorem}
\begin{proof}
Our hypotheses imply that $r \in (0,2)$. Therefore, we distinguish between two cases: either $r\ge 1$ or $r<1$. We shall focus on the case $r \in [1,2)$ because the case $r \in (0,1)$ can be dealt with the same arguments, but performed over the function $u$ instead of its gradient.

Since $r \ge 1$, it turns out that $\beta + 2s > 1$, whence $\beta > \max \{ 1-2s, 0 \}$. Moreover, we have $1/p > r - s \ge 1 - s$ and consequently we can apply Theorem \ref{thm:W1p-regularity} ($W^1_p$-regularity) to deduce $u \in \Wzop(\Omega)$. This concludes the proof in the case $r = 1$; if $r > 1$ we next aim to bound $\|D u\|_{\widetilde W^{r-1}_p(\Omega)}= \|u\|_{\widetilde W^r_p(\Omega)}$. Similarly to the proof of Theorem \ref{thm:W1p-regularity}, we split the domain of integration into the sets
\[
A = \Big\{ (x,y) \in D \colon |x-y| < \delta(x) = \delta(x,y) \Big\},
\quad
B = \big[ D \setminus A \big] \cup \big[ \Omega \times \Omega^c \big],
\]  
and write
\[
\frac{\|u\|_{\widetilde W^r_p(\Omega)}^p}{\Lambda(r-1,d,p)} = 2 \iint_{A} \frac{|D u(x)- D u(y)|^p}{|x-y|^{d+(r-1)p}} \, dy dx + 2 \iint_{B} \frac{|D u(x)- D u(y)|^p}{|x-y|^{d+(r-1)p}} \, dy dx.
\]

We now proceed as in the case $s \neq 1/2$ of Theorem \ref{thm:W1p-regularity}. On the set $A$, we exploit the bound on $|D u(x) - D u (y)|$ given in \eqref{eq:higher-beta} to obtain 
\[ \begin{split}
\iint_{A} & \frac{|D u(x)- D u(y)|^p}{|x-y|^{d+(r-1)p}} \, dy dx 
  \\
& \le C \|f\|_{C^\beta(\overline{\Omega})}^p \int_\Omega  \delta(x)^{-p(\beta+s)} \int_{B(x,\delta(x))} |x-y|^{-d+ p(\beta+2s -r)} \, dy  \, dx \\
& \le \frac{C}{p (\beta+ 2s -r)} \|f\|_{C^\beta(\overline{\Omega})}^p \int_\Omega \delta(x)^{-p (r - s)} \, dx.
\end{split}\]
because the assumption $r< \beta +2s$ ensures the convergence of the integral on $B(x,\delta(x))$.
In view of \eqref{eq:Mazya} with $\alpha = p (r - s) < 1$, the integral in the right hand side above is convergent   and is of order $(1-p (r - s))^{-1}$, whence
\[
\iint_{A}  \frac{|D u(x)- D u(y)|^p}{|x-y|^{d+(r-1)p}} \, dy dx
\le \frac{C}{p \big(\beta+ 2s -r \big) \big(1-p (r - s)\big)} \|f\|_{C^\beta(\overline{\Omega})}^p.
\]

On the set $B$, we utilize the pointwise bound $|D u|$ given in \eqref{eq:higher-beta} to write
\[
|D u(x) - D u(y)| \le |D u(x)| + |D u(y)| \le C \delta(x)^{s-1} \|f\|_{C^\beta(\overline{\Omega})}
\]
because either $(x,y)\in D\setminus A$, whence $\delta(x)\le\delta(y)$, or $(x,y)\in\Omega\times\Omega^c$ and $D u(y)=0$. Consequently, since $p(r-s)<1$,
\[\begin{aligned}
  \iint_{B} & \frac{|D u(x)- D u(y)|^p}{|x-y|^{d+(r-1)p}} \, dy dx
  \\ &\le
C \|f\|_{C^\beta(\overline{\Omega})}^p \int_\Omega  \delta(x)^{-p(1-s)} \int_{B(x,\delta(x))^c} |x-y|^{-d-(r-1)p} \, dy \, dx \\
& \le \frac{C \|f\|_{C^\beta(\overline{\Omega})}^p}{p(r-1)}  \int_\Omega \delta(x)^{-p (r - s)} \, dx
\le \frac{C}{p\big(r-1\big)\big(1-p(r-s)\big)} \|f\|_{C^\beta(\overline{\Omega})}^p.
\end{aligned}\]

Collecting the estimates on $A$ and $B$, we infer that
\[
\|u\|_{\widetilde W^r_p(\Omega)}^p \le \frac{\Lambda(r-1,d,p)}{p(r-1)} \frac{C(\Omega, s,d) (\beta+2s-1)}{\big(\beta+2s-r\big) \big(1-p (r - s)\big)} \| f \|_{C^{\beta}(\overline \Omega)}^p,
\]
and the behavior of the ratio $\frac{\Lambda(r-1,d,p)}{p(r-1)}$ is robust with respect to $r\to1$ and $p\to\infty$, according to \eqref{eq:asymptotics-C}. This implies that $u \in \widetilde W^r_p(\Omega)$ and \eqref{eq:higher-regularity} is valid.
\end{proof}

For the purposes of approximation, we aim to take $r$ as large as possible. On the one hand, we have the limitation $r < s + 1/p$ from the hypotheses of Theorem \ref{thm:higher-regularity} (Sobolev regularity). On the other hand, one requires $r > s + \frac{d}{p} - \frac{d}{2}$ in order to have $\sob (W^r_p) > \sob(H^s)$. These two straight lines $(\frac{1}{p},r)$ meet at $\frac{1}{p}=\frac{d}{2(d-1)}$, whence we deduce the extreme differentiability and integrability indices
\begin{equation}\label{optimal-differentiability}
    r = s + \frac{1}{p}, \quad p = \frac{2(d-1)}{d}.
\end{equation}
This is in agreement with the estimates in weighted spaces from \cite{AcosBort2017fractional, BoLeNo21}. Let us now specify admissible choices of differentiability parameter $r$ and integrability parameter $q$ so that $r$ is as close to $s+\frac{1}{p}$ and $q$ as close to $p$ as possible.

\begin{corollary}[optimal regularity for $d \ge 2$]\label{C:choice-parameters}
Let $p = \frac{2(d-1)}{d}$ and $\beta>0$ satisfy $\max \big\{\frac{1}{p} - s, 0 \big\} \le \beta < 2 -2s$. If $f \in C^\beta(\overline \Omega)$, then any $0 < \eps < \frac{d}{p} \big( 1 - \frac{p^2}{d} \big)$ yields
\begin{equation} \label{eq:optimal-regularity}
\|u\|_{\widetilde{W}^{s + \frac1p - \eps}_{p+\eps}(\Omega)}^{p+\eps} \le  \frac{C(\Omega, s,d)}{\big(\beta + s - \frac{1}{p} + \eps \big) \,\eps^\gamma} \| f \|_{C^{\beta}(\overline \Omega)}^{p+\eps},
\end{equation}
where $\gamma = 1$ if $d \ge 3$ and $\gamma = 2$ if $d = 2$.
\end{corollary}
\begin{proof}
We set $r = s + \frac{1}{p} - \eps$ and $q=p+\eps$ for $\eps > 0$ sufficiently small to be chosen. We first notice that $\eps > 0$ implies
\[
r - s - \frac{1}{q} = \frac{1}{p} - \frac{1}{q} + \eps = \frac{1}{p} \Big( 1 - \frac{1}{1+\frac{\eps}{p}} \Big) - \eps < 0,
\]
because $\frac{1}{1+x} > 1 - x$ for all $x>0$ and $p\ge1$. We next observe $\eps < \frac{d}{p}\big(1-\frac{p^2}{d}\big)$ yields
\[
r - s - \frac{d}{q} + \frac{d}{2} = \frac{1}{p} - \frac{d}{q} + \frac{d}{2} - \eps =
\frac{d}{p} - \frac{d}{q} - \eps
= \frac{d}{p} \Big( 1 - \frac{1}{1+\frac{\eps}{p}} \Big) - \eps > 0
\]
because $\frac{1}{1+x} < 1 - \alpha x$ for $\alpha = \frac{p^2}{d}<1$ and $0 < x < \frac{1-\alpha}{\alpha}$.

We are thus entitled to apply Theorem \ref{thm:higher-regularity} (Sobolev regularity). We see that the ranges of $\beta$ and $r$ are admissible because
\[
\beta + 2s - r = \beta + s - \frac{1}{p} + \eps \ge \eps,
\]
whereas
\[
1 - q(r-s) = 1 - (p+\eps)(r-s) = \eps \Big( p - \frac{1}{p} + \eps  \Big) \ge \eps^\gamma
\]
with $\gamma = 2$ for $d=2$ because $p=1$, and otherwise $\gamma=1$ for $d>2$ because $p>1$. Finally, applying \eqref{eq:higher-regularity} leads to \eqref{eq:optimal-regularity} and concludes the proof.
\end{proof}

\begin{remark}[parameter ranges]\label{R:parameter-ranges}
  We point out that the extreme parameters satisfy $p=\frac{2(d-1)}{d} > 1$ for $d>2$
  and $r = s + \frac{1}{p} < 1$ for $d>2$ and $s < \frac{d-2}{2(d-1)}$.
\end{remark}

\begin{remark}[regularity of $f$]
  Assuming $f$ to be more regular than $C^\beta$ for $\beta \ge \max \{\frac{1}{p} - s, 0\}$ in Theorem \ref{thm:higher-regularity} (Sobolev regularity) would entail higher regularity for the solution, namely, in $u \in W^r_p$ for some $r \ge s + \frac{1}{p}$. However, this would not be useful for our approximation purposes because our technique leads to the extreme conditions \eqref{optimal-differentiability} regardless of the smoothness of $f$. We also observe that choosing the minimal regularity $\beta=\frac{1}{p}-s \ge 0$ the denominator of \eqref{eq:optimal-regularity} contains an additional power of $\eps$.
\end{remark}

\begin{remark}[optimal choice of parameters for $d=1$]
The case $d=1$ is strikingly different from $d>1$. To see this, note that
\[
p > \frac{2}{1+2(r - s)} \quad\Rightarrow\quad
\sob(W^r_p) - \sob (H^s) = r - s - \frac{1}{p} + \frac{1}{2} > 0.
\]
We can then set $r > 0$ arbitrarily large --as long as $f$ is sufficiently smooth-- and $\frac1p =  \frac12 + r - s - \eps$ for some $\eps \in (0,\frac12)$; hence, the condition $r < s + \frac1p$ is automatically satisfied. We point out that, in this case, the optimal regularity may correspond to $p < 1$ if $f \in C^\beta(\overline \Omega)$ with $\beta > \frac12 - s$. We illustrate this in Figure \ref{fig:DeVore_dim1}.

\begin{figure}[htbp]
\begin{center}
  \includegraphics[width=0.65\linewidth]{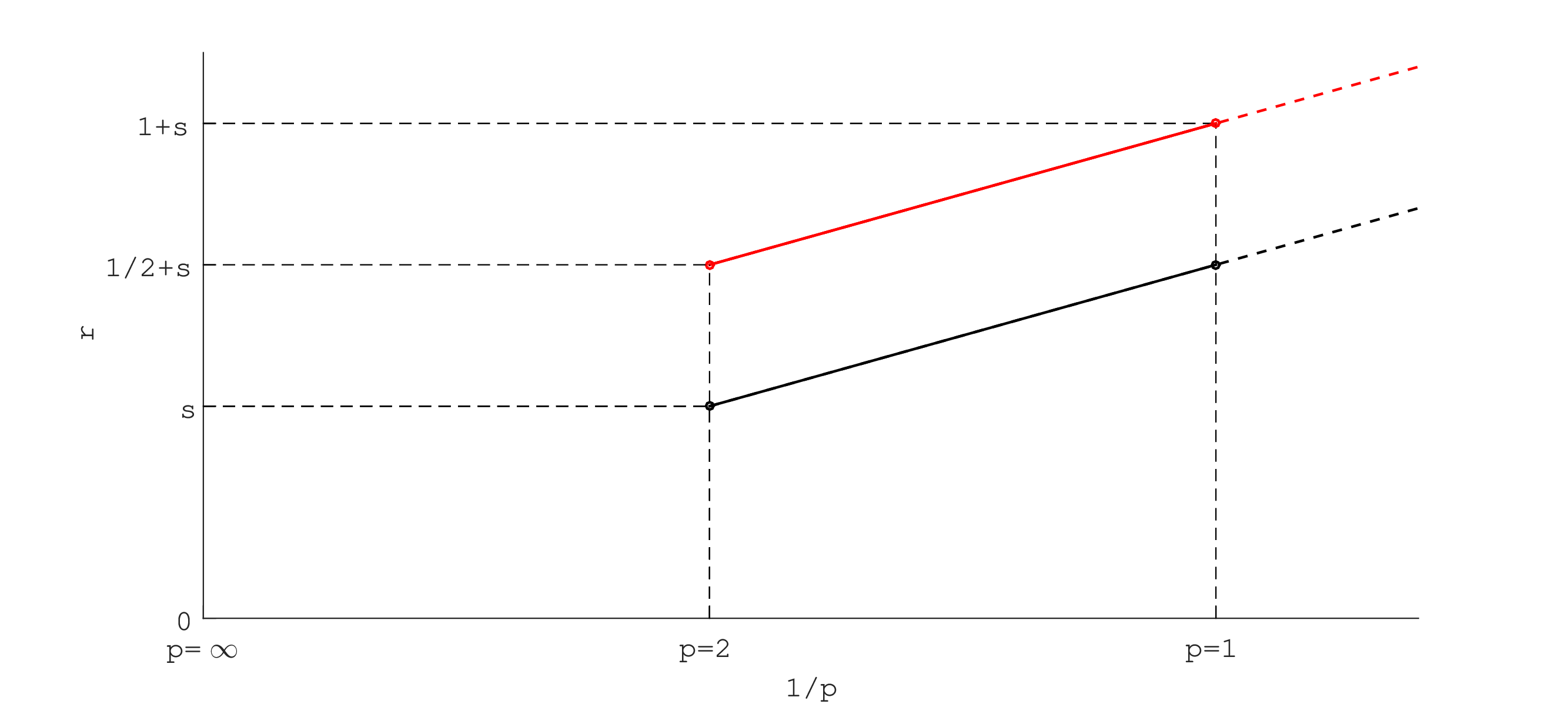}
\end{center}
\vspace{-0.5cm}
\caption{DeVore diagram in the case $d = 1$. The regularity line (red) $r = s + \frac{1}{p}$ and the Sobolev line (black) $r = s + \frac{1}{p} - \frac12$ are parallel, and therefore we can set $r$ as large as allowed by the regularity of $f$.}
\label{fig:DeVore_dim1}
\end{figure}
\end{remark}

\begin{remark}[corner singularities in two dimensions] \label{rmk:corners}
We have proved our main results for Lipschitz domains $\Omega$ provided the solution $u$ to \eqref{eq:weak} obeys \eqref{eq:Holder-regularity}, which entails the asymptotic boundary behavior \eqref{eq:bdry_behavior}. Sufficient conditions leading to \eqref{eq:Holder-regularity}, proposed in \cite{AbatangeloRosOton:2020,RosOtonSerra} and discussed after \eqref{eq:Holder-regularity}, are severe geometric restrictions that imply convexity of $\Omega$ if $\Omega$ is polyhedral. If $\Omega$ is polygonal for $d=2$, possibly with reeentrant corners, Gimperlein, Stephan and Stocek \cite{Gimperlein:19} prove that \eqref{eq:bdry_behavior} is also valid near edges, but that a different expansion holds near reentrant corners: at a vertex $V$, one has generically
\[
u(x) \approx \mbox{dist}(x,V)^\lambda,
\]
where $\lambda$ is the smallest eigenvalue of a perturbation of the Laplace-Beltrami operator on the upper hemisphere with mixed boundary conditions. One has the bound $\lambda > \max \{ 0 , s - 1/2 \}$, which is attained when the vertex angle tends to $2\pi$.

Therefore, one can perform an heuristic argument similar to that in the introduction, take derivatives of order $r$ of the function $v(x) = |x|^\lambda$, and deduce that
\begin{equation}\label{eq:corner-differentiability}
  |\partial^r v (x)| \approx |x|^{\lambda-r}
  \quad\Rightarrow\quad
  r < \lambda + \frac{2}{p}
\end{equation}
gives the differentiability index $r$ near a reentrant corner for $p$-integrability. Since $\lambda > s - 1/2$, we conclude that for $p \le 2$ the extreme differentiability index in \eqref{optimal-differentiability} 
\[
r = s + \frac{1}{p} = s-\frac{1}{2} + \frac{1}{p} + \frac{1}{2} < \lambda + \frac{2}{p}
\]
is less stringent than \eqref{eq:corner-differentiability}. It thus follows from this heuristic discussion that reentrant corners do not have significant effects on the approximability of problem \eqref{eq:Dirichlet} on polygonal domains in $\R^2$. Edge singularities dominate corner ones!
\end{remark}

\section{Adaptive Construction of Graded Meshes}\label{S:graded-meshes}

We now consider the approximation of the Dirichlet problem \eqref{eq:Dirichlet} on a polyhedral
bounded domain $\Omega\subset\R^d$ with $d\ge 2$.
Let $\T=\{T\}$ be shape-regular conforming meshes made of closed simplices $T$ that cover $\overline\Omega$ exactly, namely
\[
\sigma := \sup_{\T} \max_{T \in \T} \frac{h_T}{\rho_T} <\infty,
\]
where $h_T=|T|^{1/d}$ is proportional to
the diameter of $T$ and $\rho_T $ is the diameter of the largest ball contained in $T$.
Let $\V_\T$ denote the space of continuous piecewise linear functions over $\T$
\begin{equation*}\label{eq:space}
\V_\T := \big\{v\in C^0(\overline{\Omega}): \quad v|_T \in \poly^1 \quad\forall T\in\T  \big\}.
\end{equation*}
Let $u_\T\in\V_\T$ be the Galerkin approximation of $u\in\wt{H}^s(\Omega)$ given by \eqref{eq:weak},
namely
\begin{equation}\label{eq:galerkin}
u_\T\in\V_T: \quad a(u_\T,v) = \langle f , v \rangle \quad\forall v \in\V_T.
\end{equation}
Our goal is to construct {\it graded} meshes $\T=\{T\}$ that deliver
an optimal convergence rate for the energy error in terms of the cardinality $\#\T$ of $\T$
to compute $u_\T$:
\begin{equation*}\label{eq:optimal-rate}
\| u - u_\T \|_{\wt{H}^s(\Omega)} \lesssim (\#\T)^{-\alpha} \|f\|_{C^\beta(\overline{\Omega})}.
\end{equation*}
We adopt $\#\T$ as a measure of complexity to compute $u_\T$ so the question is to find the
largest value of $\alpha$ compatible with the regularity of $u$ and shape regularity of $\T$.
The latter entails dealing exclusively with isotropic graded meshes. We will
comment on the limitations of this choice as well as the existing theory.

\subsection{Localization.}
In view of \eqref{eq:weak} and \eqref{eq:galerkin}, Galerkin orthogonality $a(u-u_\T,v) = 0$ holds for all $v \in\V_T$ and $u_\T$ satisfies the best approximation property
\begin{equation}\label{eq:best-approx}
\| u - u_\T \|_{\wt{H}^s(\Omega)} = \min_{v\in\V_\T} \| u - v \|_{\wt{H}^s(\Omega)} .
\end{equation}
Therefore, we focus on estimating the right-hand side of \eqref{eq:best-approx}, where
$v=\Pi_\T u \in\V_\T$ will be a suitable local quasi-interpolant of $u$, e.g. \cite{ScottZhang:1990}.
To localize the nonlocal
seminorm $|\cdot|_{\tHs}$, we first define the star (or patch) of a set  $A \in \Omega$ by
\[
  S_A^1 := \bigcup \big\{ T \in \T \colon T \cap A \neq \emptyset \big\}.
\]
Given $T \in \T$, the star $S_T^1$ of $T$ is the first ring of $T$ and the star $S_T^2$ of $S_T^1$
is the second ring of $T$. The following localized estimate is due to B. Faermann
\cite{Faermann2, Faermann}
\begin{equation} \label{eq:localization-F}
  |v|_{H^s(\Omega)}^2 \leq \sum_{T \in \T} \left[ \int_T \int_{S_T^1} \frac{|v (x) - v (y)|^2}{|x-y |^{d+2s}} \; dy \; dx + \frac{C(d,\sigma)}{s h_T^{2s}} \| v \|^2_{L^2(T)} \right]
\end{equation}
for all $ v \in H^s(\Omega)$.
This inequality shows that to estimate fractional seminorms over $\Omega$, it suffices to
compute integrals over the set of patches $\{T \times S_T^1 \}_{T \in \T}$ plus local zero-order contributions.
For our purposes, we need the following variant suited for the zero-extension spaces $\wt{H}^s(\Omega)$, which relies on the {\it extended} stars
\begin{equation*}
\wt{S}_T^1 :=
\left\lbrace \begin{array}{rl}
  S_T^1   & \textrm{if } T \cap \partial\Omega = \emptyset,
  \\
  B_T    & \textrm{otherwise,}
\end{array}\right.
\end{equation*}
where $B_T=B(x_T,Ch_T)$ is the ball of center $x_T$ and radius $Ch_T$, with $x_T$ being the barycenter of $T$, $h_T=|T|^{1/d}$ and $C=C(\sigma)$ a shape regularity dependent constant such that $S_T^1\subset B_T$. The extended second ring $\wt{S}_T^2$ of $T$ is given by
\[
\wt{S}_T^2 := \bigcup \big\{ \wt{S}_{T'}^1 \colon T' \in \T, \mathring{T}' \cap S_T^1 \neq \emptyset \big\}.
\]

\begin{lemma}[localization of the $\tHs$ seminorm]
Let $\T$ be a shape-regular triangulation of $\Omega$. Then, for all $v \in \tHs$ it holds that
\begin{equation} \label{eq:localization-F-tilde}
  \|v\|_{\tHs}^2 \leq \sum_{T \in \T} \left[ \int_T \int_{\widetilde S_T^1} \frac{|v (x) - v (y)|^2}{|x-y |^{d+2s}} \; dy \; dx + \frac{C(d,\sigma)}{s h_T^{2s}} \| v \|^2_{L^2(T)} \right].
\end{equation}
\end{lemma}
\begin{proof}
Let $v \in \tHs$. In view of \eqref{eq:localization-F} and the expression
\[
\|v\|_{\tHs}^2 = |v|_{H^s(\Omega)}^2 + 2 \iint_{\Omega \times \Omega^c} \frac{|v(x)-v(y)|^2}{|x-y|^{d+2s}}  \; dy \; dx,
\]
it suffices to bound the last integral on the right hand side. We first split the domain of integration into pieces $T\cap\Omega^c$ for all $T \in \T$, and distinguish two cases depending on the location of $T$ relative to $\partial\Omega$. If $\overline T \cap \pp\Omega \neq \emptyset$, we replace $\Omega^c$ by $B_T$ and $B_T^c\cap\Omega^c$ and note that $B_T=\wt{S}_T^1$ contributes to the first term in \eqref{eq:localization-F-tilde}. The integral over $B_T^c\cap\Omega^c$, instead, is similar to that over $\Omega^c$ for $T\cap\partial\Omega=\emptyset$ in that the domain of integration is contained in $B(x_T,Ch_T)^c$. Since $v=0$ in $\Omega^c$, we see that
\[
\iint_{T \times \Omega^c} \frac{|v(x)-v(y)|^2}{|x-y|^{d+2s}}  \; dy \; dx \lesssim \int_T |v(x)|^2  \int_{Ch_T}^\infty \rho^{-1-2s}  \; d\rho \; dx \le \frac{C}{s h_T^{2s}} \| v \|^2_{L^2(T)},
\]
where $C=C(d,\sigma)$. This contributes to the second term in \eqref{eq:localization-F-tilde} and concludes the proof.
\end{proof}

In either \eqref{eq:localization-F} or \eqref{eq:localization-F-tilde}, if the $L^2$-contributions have vanishing means over elements --as is often the case whenever $v$ is an interpolation error-- a Poincar\'e inequality allows one to estimate them in terms of local $H^s$-seminorms. We consider the Scott-Zhang quasi-interpolation operator
\[
\Pi_\T:W^r_p(\Omega)\to\V_\T,
\quad
r>1/p,
\]
that preserves the vanishing trace within the subspace $\Wzrp(\Omega)$ \cite{ScottZhang:1990}, and extends by zero to $\Omega^c$ thereby keeping approximation properties in $\wt{W}^r_p(\Omega)$.
Therefore, one can prove the following local quasi-interpolation estimates (see, for example, \cite{AcosBort2017fractional,BoNoSa18,CiarletJr}).

\begin{lemma}[local interpolation estimates] \label{L:app_SZ}
Let $T \in \T$, $s \in (0,1)$, $r \in (s, 2]$, $pr>1$, and $\SZ:W^r_p(\Omega)\to\V_\T$ be the Scott-Zhang
quasi-interpolation operator. If $v \in W^r_p (\wt{S}_T^2)$, then
\begin{equation} \label{eq:interpolation}
  \int_T \int_{\widetilde S_T^1} \frac{|(v-\SZ v) (x) - (v-\SZ v) (y)|^2}{|x-y|^{d+2s}} \, d y \, d x \le C \, h_T^{2\l}
  |v|_{W^r_p(\wt{S}_T^2)}^2,
\end{equation}
where $\l = \sob(W^r_p) - \sob(H^s) = r -s - d \big(\frac{1}{p} - \frac{1}{2} \big) > 0$
and $C = C(\Omega,d,s,\sigma,r)$.
\end{lemma}

Combining \eqref{eq:localization-F-tilde} and \eqref{eq:interpolation} with local $L^2$-error estimates
for the Scott-Zhang quasi-interpolation operator $\SZ$, we deduce localized error estimates.

\begin{proposition}[localized error estimates]\label{P:error-estimate}
Let $s \in (0,1)$, $r \in (s, 2]$, $rp>1$, and $\SZ: W^r_p(\Omega)\to\V_\T$ be the Scott-Zhang
quasi-interpolation operator. If $v \in \widetilde W^r_p (\Omega)$, then
\begin{equation*}\label{eq:error-estimate}
\|v - \SZ v\|_{\tHs} \leq  C \left(\sum_{T \in \T} h_T^{2\l} |v|_{W^r_p(\wt{S}_T^2)}^2\right)^{\frac12},
\end{equation*}
where $\l = r -s - d \big(\frac{1}{p} - \frac{1}{2} \big) > 0$
and $C = C(\Omega,d,s,\sigma, r)$.
\end{proposition}

\subsection{Graded Bisection Meshes}
%
We now briefly discuss the {\it bisection method}, the most elegant
and successful technique for subdividing $\Omega$ in any dimension $d$ into a conforming
mesh $\T$. Every simplex $T\in\T$ must have an edge $e(T)$ marked for refinement which is
used as follows to subdivide $T$ into two children  $T_1,T_2$ such that $T = T_1 \cup T_2$:
connect the mid point of
$e(T)$ with the $d-2$ vertices of $T$ that do not lie on $e(T)$ and marked suitable edges $e(T_1), e(T_2)$
of the children for further refinement. This procedure is called
{\it newest vertex bisection} in dimension $d=2$. If all the simplices sharing
$e(T)$ have this edge marked for refinement, then we have a compatible patch and
bisection creates a conforming refinement of $\T$: this step is completely local in that
the refinement does not propagate beyond the patch. Otherwise, at least one element in the patch has
an edge other than $e(T)$ marked for refinement, then the refinement procedure
must go outside the patch to maintain conformity (nonlocal step). Therefore, two natural questions arise:
\begin{enumerate}[$\bullet$]
\item {\it Completion}:
How many elements other than $T$ must be refined to keep the mesh conforming?  
  
\item {\it Termination}:   
Does this procedure terminate?
\end{enumerate}
To guarantee termination, a special labeling of the initial mesh $\T_0$ is required
(choice of edge $e(T)$ for each element $\T_0$).
Completion is rather tricky to assess and was done by
P. Binev, W. Dahmen and R. DeVore for $d=2$ \cite{BinevDahmenDeVore:2004}
and R. Stevenson for $d>2$ \cite{Stevenson:2008}; we refer to the
surveys \cite{NochettoSiebertVeeser:2009,NochettoVeeser:2012} for a rather complete
discussion.

Given the $j$-th refinement $\T_j$ of $\T_0$ and a subset $\M_j\subset\T_j$ of elements marked
for bisection, the procedure
\[
\T_{j+1} = \REFINE (\T_j,\M_j)
\]
creates the smallest conforming refinement $\T_{j+1}$ of $\T_j$ upon bisecting all elements of
$\M_j$ at least once and perhaps additional elements to keep conformity. We point out that
it is simple to construct counterexamples to the estimate
\[
\#\T_{j+1} - \#\T_j \le \Lambda \; \#\M_j
\]
where $\Lambda$ is a universal constant independent of $j$; see \cite[Section 1.3]{NochettoVeeser:2012}.
However, this can be repaired upon considering the cumulative effect of a sequence of conforming
bisection meshes $\{\T_j\}_{j=0}^k$ for any $k$. In fact, the following weaker estimate is valid and is due to
P. Binev, W. Dahmen and R. DeVore for $d=2$ \cite{BinevDahmenDeVore:2004}
and R. Stevenson for $d>2$ \cite{Stevenson:2008}; see also
\cite{NochettoSiebertVeeser:2009,NochettoVeeser:2012}.

\begin{lemma}[complexity of bisection]\label{L:complexity-bisection}
  There is a universal constant $\Lambda_0>0$, which depends on $\T_0$ and its labelling as well as $d$,
  such that for all $k\ge1$
  \[
  \#\T_k - \#\T_j \le \Lambda_0 \sum_{j=0}^{k-1} \#\M_j.
  \]
\end{lemma}

Consequently, the cardinality increase for the entire refinement process is controlled by
the total number of marked elements despite that fact that this is false for single
refinement steps. The latter may contain large chains of elements, perhaps of all levels
and reaching the boundary, but these events do not happen very often. This illustrates
the subtle character of the proof of Lemma \ref{L:complexity-bisection}.

\subsection{Adaptive Approximation}
%
We examine now the construction of graded bisection meshes that deliver quasi-optimal
approximation rates. The benefits of graded meshes are well known in the finite element
literature. Necessary conditions for their design are given by P. Grisvard \cite{Grisvard}
for second order problems in polygonal domains. Reference \cite{AcosBort2017fractional} applied
similar ideas to the Dirichlet integral fractional Laplacian given in \eqref{eq:Dirichlet};
see also \cite{ainsworth2017aspects, BBNOS18, BoLeNo21, gimperlein2019space, Lischke}. This approach consists
of deriving the desirable size of isotropic elements depending on the distance to
singularities and next counting elements. This assumes that it is possible to construct
such graded meshes a priori, or in other words that there are no geometric obstructions
to placing elements of varying size together to cover the entire domain $\Omega$ in
dimension $d>1$; the situation is trivial for $d=1$.

The purpose of this section is to give a constructive procedure for isotropic graded meshes
produced by the bisection algorithm for \eqref{eq:Dirichlet}. Our proof is inspired by DeVore et al. 
\cite[Proposition 5.2]{BDDP:2002}; we also refer to the surveys \cite[Theorem 12]{NochettoSiebertVeeser:2009}
and \cite[Theorem 3]{NochettoVeeser:2012}. In view of Corollary \ref{C:choice-parameters} (optimal regularity for $d\ge2$) and Proposition \ref{P:error-estimate} (localized error estimates), we
introduce now a quantity that dominates the local $H^s$-error of $u\in\wt{H}^s(\Omega)$,
solution of \eqref{eq:weak}:
\begin{equation}\label{eq:local-indicator}
  E_T(u) := C_0 h_T^\l R_T(u),
  \quad R_T(u) := |u|_{W^{s+\frac{1}{p}-\eps}_{p+\eps}(\wt{S}_T^2)},
  \quad p=\frac{2(d-1)}{d}
\end{equation}
where $C_0$ depends on the shape regularity constant $\sigma$ and
\begin{equation*}\label{eq:value-r}
  \l = \frac{1}{p}-\eps- \frac{d}{p+\eps} + \frac{d}{2} > 0;
\end{equation*}
note that $\l=0$ for $\eps=0$.
It is convenient to introduce a positive lower bound on $\l$ for $\eps>0$ to simplify the calculations below.
To this end, we exploit the fact that $\frac{1}{1+x}<1-\frac{3}{4}x$ for $0<x<\frac13$ to obtain
\begin{equation}\label{eq:p+eps}
\frac{d}{p+\eps} = \frac{d}{p\big(1+\frac{\eps}{p}\big)} < \frac{d}{p} \Big(1-\frac{3\eps}{4 p} \Big)
\quad\Rightarrow\quad
\l > \eps \Big( \frac{3d^3}{16(d - 1)^2}-1 \Big) >0,
\end{equation}
provided
\[
\eps < \frac{2(d-1)}{3d};
\]
Proposition \ref{P:error-estimate}, in conjunction with local $L^2$-interpolation estimates, implies
\begin{equation}\label{eq:localization}
\|u - \Pi_\T u\|_{\tHs}^2 \le C \sum_{T\in\T} E_T(u)^2
\end{equation}

Given a tolerance $\delta>0$ and a conforming mesh $\T_0$, the following algorithm finds
a conforming refinement $\T$ of $\T_0$ by bisection
such that $E_T(u)\le\delta$ for all $T\in\T$: let $\T=\T_0$ and
\begin{algotab}
  \> $\GREEDY (\T,\delta)$\\
  \>  while $\M :=\{T\in\T : \,E_T(u) > \delta\}\ne\emptyset$\\
  \> \> $\T := \REFINE(\T,\M)$\\
  \> end while\\
  \> return($\T$)
\end{algotab}
The heuristic idea behind $\GREEDY$ is to equidistribute the local $H^s$-errors $E_T(u)$. If we further
assume that, upon termination, $E_T(u)\approx\delta$ then we immediately deduce
$\| u - \Pi_\T u \|_{\tHs}^2 \le \delta^2 (\#\T)$ as well as
\[
\delta^{p+\eps} (\#\T) \approx \sum_{T\in\T} E_T(u)^{p+\eps}
\lesssim  \|u\|_{\widetilde W^{s+\frac{1}{p}-\eps}_{p+\eps}(\Omega)}^{p+\eps}.
\]
Combining these two expressions yields
\[
\| u - \Pi_\T u \|_{H^s(\Omega)} \lesssim (\#\T)^{\frac12 - \frac{1}{p+\eps}}
\|u\|_{\widetilde W^{s+\frac{1}{p}-\eps}_{p+\eps}(\Omega)}.
\]
Our next result confirms that this heuristics is correct and chooses $\eps$ judiciously.

\begin{theorem}[quasi-optimal error estimate for $d \ge 2$]\label{T:greedy}
If $u\in \widetilde W^{s+\frac{1}{p}-\eps}_{p+\eps}(\Omega)$ with $p=\frac{d}{2(d-1)}$ satisfies
\eqref{eq:optimal-regularity} with $\beta>0$ and $\max\big\{\frac{1}{p}-s,0\big\} \le \beta < 2 -2s$,  then
Algorithm $\GREEDY$ terminates in finite steps. The resulting isotropic mesh $\T$ satisfies
\begin{equation}\label{eq:error-est}
  \| u - \Pi_\T u \|_{\tHs} \lesssim \big(\#\T\big)^{-\frac{1}{2(d-1)}}
  |\log \#\T|^{\gamma} \|f\|_{C^\beta(\overline{\Omega})},
\end{equation}
where $\gamma=2+\kappa$ if $d=2$ and $\gamma=1+\kappa$ if $d\ge 3$, with $\kappa=0$ if $\beta > \frac{1}{p}-s$ and
$\kappa=1$ if $\beta = \frac{1}{p}-s$.
\end{theorem}
\begin{proof}
We proceed in several steps.

\medskip\noindent
    {\it Step 1: Termination}. Since the local error $E_T(u)$ satisfies \eqref{eq:local-indicator},
    and bisection of $T\in\T$ reduces its size $h_T=|T|^{1/2}$ by a factor $2^{-1/d}$, $\GREEDY$
    terminates in finite number of steps $k\ge1$  depending $\T_0$ and $u$.

    In order to count the total number of marked elements $\M = \cup_{j=0}^{k-1} \M_j$, we
    organize them by size. Let $\P_j$ be the subset of $\M$ of elements with measure
    \[
    2^{-(j+1)} \le |T| < 2^{-j}
    \quad\Rightarrow\quad
    2^{-(j+1)/d} \le h_T < 2^{-j/d}.
    \]
\medskip\noindent
    {\it Step 2: Cardinality bound 1.}
    We observe that all $T$'s in $\P_j$ are \emph{disjoint}. This is because if $T_1,\,T_2\in \P_j$
    and $\mathring{T}_1\cap\mathring{T}_2\neq\emptyset$, then one of them is contained in the other,
    say $T_1\subset T_2$, due to the bisection procedure. Thus
    \[
	|T_1|\le\frac{1}{2}\,|T_2|
    \]
    thereby contradicting the definition of ${\mathcal P}_j$. This implies 
    \begin{equation*}\label{nv-small}
    2^{-(j+1)}\,\#\P_j \le |\Omega|
    \quad\Rightarrow\quad
    \#\P_j\le|\Omega|\, 2^{j+1}.
    \end{equation*}
    This $\delta$-independent bound is useful for large elements $T\in\M$, or equivalently for small
    values of $j$.
    
\medskip\noindent
    {\it Step 3: Cardinality bound 2.}
    We now deal with elements of relatively small size and seek a bound depending on $\delta$.
    In light of \eqref{eq:local-indicator}, we have for $T\in\P_j$
\[
	\delta \le E_T(u) \le C_0 2^{-j\l/d} R_T(u).
\]
Therefore, exploiting the summability of $\{R_T(u)\}_{T\in\T}$ in $\ell^{p+\eps}$, namely
\[
\sum_{T\in\T} R_T(u)^{p+\eps} \le C R(u)^{p+\eps},
\quad
R(u) := \|u\|_{\widetilde W^{s+\frac{1}{p}-\eps}_{p+\eps}(\Omega)},
\]
where $C$ accounts for the finite overlapping property of the sets $\wt{S}_T^2$, yields
\[
\delta^{p+\eps}\,\#{\mathcal P}_j \le C_0^{p+\eps} 2^{-j\l(p+\eps)/d}
\sum_{T\in \P_j} R_T(u)^{p+\eps} \le
C 2^{-j\l(p+\eps)/d} \, R(u)^{p+\eps}
\]
whence
\begin{equation*}\label{nv-large}
	\#{\mathcal P}_j\le C \big(\delta^{-1} \, R(u) \big)^{p+\eps} \,2^{-j\l(p+\eps)/d}.
\end{equation*}

\medskip\noindent
    {\it Step 4: Counting argument.} The total number of marked elements satisfies
    \[
    \#\M = \sum_{j=0}^{k-1} \#\P_j \le \sum_{j=0}^{k-1}
    \min \Big\{ |\Omega|\, 2^{j+1},  C 2^{-j\l(p+\eps)/d} \, \big( \delta^{-1} R(u) \big)^{p+\eps} \Big\}.
    \]
    Let $j_0\ge0$ be the smallest integer $j$ such that the second term dominates the first one.
    Since $\l>0$, this implies
    \[
    2|\Omega| 2^{j_0} \le C \left( \frac{R(u)}{\delta} \right)^{p+\eps} 2^{-j_0 \l (p+\eps)/d}
    < C \left( \frac{R(u)}{\delta} \right)^{p+\eps},
    \]
    whence
    \[
    |\Omega| \sum_{j=1}^{j_0} 2^j < 2|\Omega| 2^{j_0},
    \quad
    \left( \delta^{-1} R(u) \right)^{p+\eps}\sum_{j=j_0+1}^{k-1} 2^{-j \l (p+\eps)/d}
	\le C \left( \delta^{-1} R(u) \right)^{p+\eps} .
    \]
    Applying Lemma \ref{L:complexity-bisection} (complexity of bisection), we thus deduce
    \[
    \#\T - \#\T_0 \le \Lambda_0 \#\M \le C \left( \frac{R(u)}{\delta} \right)^{p+\eps}.
    \]
    If we further assume $\#\T \ge \Lambda_1 \#\T_0$, then
    $\#\T-\#\T_0 \ge \frac{\Lambda_1 - 1}{\Lambda_1}\#\T$
    and
    \[
    \delta \lesssim \big(\#\T - \#\T_0\big)^{-\frac{1}{p+\eps}} R(u) \lesssim
    \big(\#\T\big)^{-\frac{1}{p+\eps}} R(u).
    \]

     \medskip\noindent
    {\it Step 5: Error estimate.}   
    Upon termination of $\GREEDY$ we have $E_T(u)\le\delta$ for all $T\in\T$, whence
    \[
    \|u-\Pi_\T u\|_{\tHs}^2 \lesssim \delta^2 \, \#\T
    \lesssim \big(\#\T\big)^{1-\frac{2}{p+\eps}} R(u)^2.
    \]
    This estimate confirms the heuristic discussion prior to this theorem. It remains to choose
    the parameter $\eps$ that so far has been small but free. We see that
    \[
    1 - \frac{2}{p+\eps} > 1-\frac{2}{p} + \frac{3}{2p^2}\eps
    = - \frac{1}{d-1} + \frac{3d^2}{8(d-1)^2} \eps
    \]
    by virtue of \eqref{eq:p+eps}. On the other hand, the optimal regularity estimate \eqref{eq:optimal-regularity} gives
    \[
    R(u) = \|u\|_{\widetilde W^{s+\frac{1}{p}-\eps}_{p+\eps}(\Omega)} \lesssim \frac{1}{\eps^\gamma} \|f\|_{C^\beta(\overline{\Omega})}.
    \]
    Combining these two expressions implies
    \[
    \|u-\Pi_\T u\|_{\tHs} \lesssim \big( \#\T \big)^{-\frac{1}{2(d-1)}}
      \frac{\big( \#\T \big)^{\frac{3d^2}{16(d-1)^2}\eps}}{\eps^\gamma} \|f\|_{C^\beta(\overline{\Omega})},
    \]
    and noticing that the minimum of the function $\eps\mapsto \eps^{-\gamma} N^{\alpha\eps}$
    is attained at $\eps = \frac{\gamma}{\alpha |\log N|}$ leads to the asserted estimate \eqref{eq:error-est}.
\end{proof}

\begin{remark}[logarithmic factor]\label{R:logarithm}
It is worth realizing that the presence of the logarithmic factor in \eqref{eq:error-est} is due to the lack of uniform regularity $u \in \wt{W}^{s+1/p-\eps}_{p+\eps}(\Omega)$ with respect to $\eps$ in Corollary \ref{C:choice-parameters} (optimal regularity for $d \ge 2$). The latter is an intrinsic property of weak solutions $u$ of the Dirichlet fractional Laplacian \eqref{eq:weak} associated with their boundary behavior \eqref{eq:bdry_behavior}.
\end{remark}

\begin{remark}[comparison with weighted estimates]
Theorem \ref{T:greedy} shows that the convergence rate in the energy norm of the $\GREEDY$ algorithm is $\frac{1}{2(d-1)}$ (up to logarithmic factors). In contrast, an a priori finite element analysis on quasi-uniform meshes \cite{AcosBort2017fractional} yields an order of convergence $\frac{1}{2d}$. Reference \cite{AcosBort2017fractional} derives a priori estimates on graded meshes based on regularity estimates in weighted Sobolev spaces, and obtains the same convergence rate as in Theorem \ref{T:greedy} but the proof is not constructive.
\end{remark}

\begin{remark}[anisotropic approximation]\label{R:anisotropic}
For dimension $d \ge 2$, continuous piecewise linear Lagrange interpolation in $\widetilde H^s(\Omega)$ delivers the optimal convergence rate $\frac{2-s}{d}$ on shape-regular meshes. Since the boundary layer singularity \eqref{eq:bdry_behavior} reduces the rate to \eqref{eq:error-est} for isotropic elements, namely
\[
\frac{1}{2(d-1)} < \frac{2-s}{d},
\]
one may argue that using anisotropic elements could improve upon this rate.
  This endeavor is well understood for the classical Laplacian on
  polyhedral domains. This requires the following ingredients:
  \begin{enumerate}[$\bullet$]
  \item
    {\it Anisotropic regularity}: Precise characterization of the tangential and normal regularity of
    the solution $u\in\wt{H}^s(\Omega)$ near the boundary $\partial\Omega$; see \cite{GuoBabuska:1997,MirebeauCohen:2010} for $s=1,0$.

  \item
    {\it Localized anisotropic error estimates}: Anisotropic counterpart of the local fractional
    estimates \eqref{eq:localization} and \eqref{eq:interpolation}; see \cite{SchotzauSchwabWihler:2013} for $s=1, d=3$.

  \item
    {\it Construction of anisotropic meshes $\T$}: Suitable anisotropic mesh generation with
    optimal complexity to replace the bisection algorithm; see \cite{MirebeauCohen:2012}
      for $s=0, d=2$.
  \end{enumerate} 
  These topics are open for fractional Sobolev spaces even for second-order problems.
\end{remark}

\section{Numerical Experiments}\label{S:numerical-experiments}
For problems posed on Lipschitz domains with solutions $u$ satisfying \eqref{eq:Holder-regularity}, we have proved that the optimal regularity $u \in \widetilde W^{s+1/p-\eps}_{p+\eps}(\Omega)$, $p = \frac{2(d-1)}{d}$ gives rise to optimal convergence rates in shape-regular meshes. Moreover, in Remark \ref{rmk:corners} (corner singularities in two dimensions) we gave an heuristic argument demonstrating that the same regularity properties are valid on arbitrary polygonal domains in $d=2$ dimensions. We now illustrate this behavior by performing some numerical experiments on the $L$-shaped domain $\Omega = (-1,1)^2 \setminus ([0,1) \times (-1,0])$. We solve \eqref{eq:Dirichlet} with $f=1$ and $s \in \{ 0.25, 0.5, 0.75\}$ with continuous, piecewise linear finite elements. We used the MATLAB code from \cite{ABB} to assemble the resulting stiffness matrices, and performed an adaptive mesh refinement algorithm with a greedy marking strategy based on the package provided in \cite{Funken:11}.

Our discussion in Section \ref{S:graded-meshes} suggests the use of the quantity $E_T(u)$ in \eqref{eq:local-indicator} as an error estimator. However, estimating $W^{s+1/p-\eps}_{p+\eps}$-seminorms on stars is computationally expensive; instead, we revisit the proof of Theorem \ref{thm:higher-regularity} (Sobolev regularity) to obtain upper bounds for these local seminorms.
Heuristically, let us assume $T \in \T$ is such that $d_T:= \dist(\wt{S}^2_T, \pp \Omega) > 0$, so that $\wt{S}^2_T = S^2_T$. It follows by shape-regularity that $\delta(x) \approx d_T$ for all $x \in S_T^2$. Invoking the same argument we used to bound the integral on the set $A$ in the proof of Theorem \ref{thm:higher-regularity}, assuming $f \in C^\beta(\overline \Omega)$ with $\beta > 1-s$ and recalling that for $d=2$ we have $p = \frac{2(d-1)}{d} = 1$, we obtain
\[
|u|_{W^{s+1-\eps}_{1+\eps}(S^2_T)}^{1+\eps} \le C(\sigma) \, h_T^{(1+\eps)(\beta+s-1+\eps)} \, d_T^{-\beta-s}  \, |T| .
\]
Because $d_T \ge C h_T$, we can write 
\[
|u|_{W^{s+1-\eps}_{1+\eps}(S^2_T)}^{1+\eps} \le C(\sigma) \, h_T^{\eps(\beta+s+\eps)} \, d_T^{-1} \, |T| ,
\]
and therefore we roughly have
\[
E_T(u) \le C d_T^{-1} |T|
\]
for $E_T(u)$ defined in \eqref{eq:local-indicator}. We propose the computable surrogate error estimator
\begin{equation}\label{eq:estimator}
\mathcal{E}_T(u) := |T| \, \dist (x_T, \pp\Omega)^{-1} \quad\forall \, T \in \T,
\end{equation}
where  $x_T$ is the barycenter of $T$; $\mathcal{E}_T(u)$ is thus well defined even for $T$'s whose extended second ring $\wt{S}_T^2$ touch $\partial\Omega$. We point out that
\begin{equation}\label{eq:subadditivity}
\sum_{T\in\T}  \mathcal{E}_T(u) \lesssim \log \#\T,
\end{equation}
which reveals the subadditive character of $\mathcal{E}_T(u)$. To see this, we first note that for $T\in\T^0$ that do not touch $\partial\Omega$ we have $\mathcal{E}_T(u) \lesssim \int_T \dist(x,\pp\Omega)^{-1} dx$, whence
\[
\sum_{T\in\T^0} \mathcal{E}_T(u) \lesssim \int_{C h_\textrm{min}}^{\textrm{diam} (\Omega)} \frac{d\rho}{\rho}
\lesssim |\log h_{\textrm{min}}| \lesssim \log \#\T.
\]
Next, for $T\in\T^\partial$ touching $\partial\Omega$ we have $\dist (x_T, \pp\Omega)\approx h_T$,
whence $\mathcal{E}_T(u) \lesssim h_T$ and
\[
\sum_{T\in\T^\partial} \mathcal{E}_T(u) \lesssim \sum_{T\in\T^\partial} h_T \lesssim |\partial\Omega|.
\]
We subordinate the $\GREEDY$ algorithm to the surrogate estimator $\mathcal{E}_T(u)$. The preceding subadditivity property of $\mathcal{E}_T(u)$ guarantees that Step 3 (cardinality bound 2) of the proof of Theorem \ref{T:greedy} (quasi-optimal error estimate for $d \ge 2$) is still valid for $\mathcal{E}_T(u)$ and that $\GREEDY$ delivers an error estimate similar to \eqref{eq:error-est}. We stress that the presence of the logarithmic factor in \eqref{eq:subadditivity}, due to the lack of uniform summability of $\mathcal{E}_T(u)$, is consistent with Remark \ref{R:logarithm} (logarithmic factor).

\begin{figure}[htbp]
\begin{center}
 \includegraphics[width=0.34\linewidth]{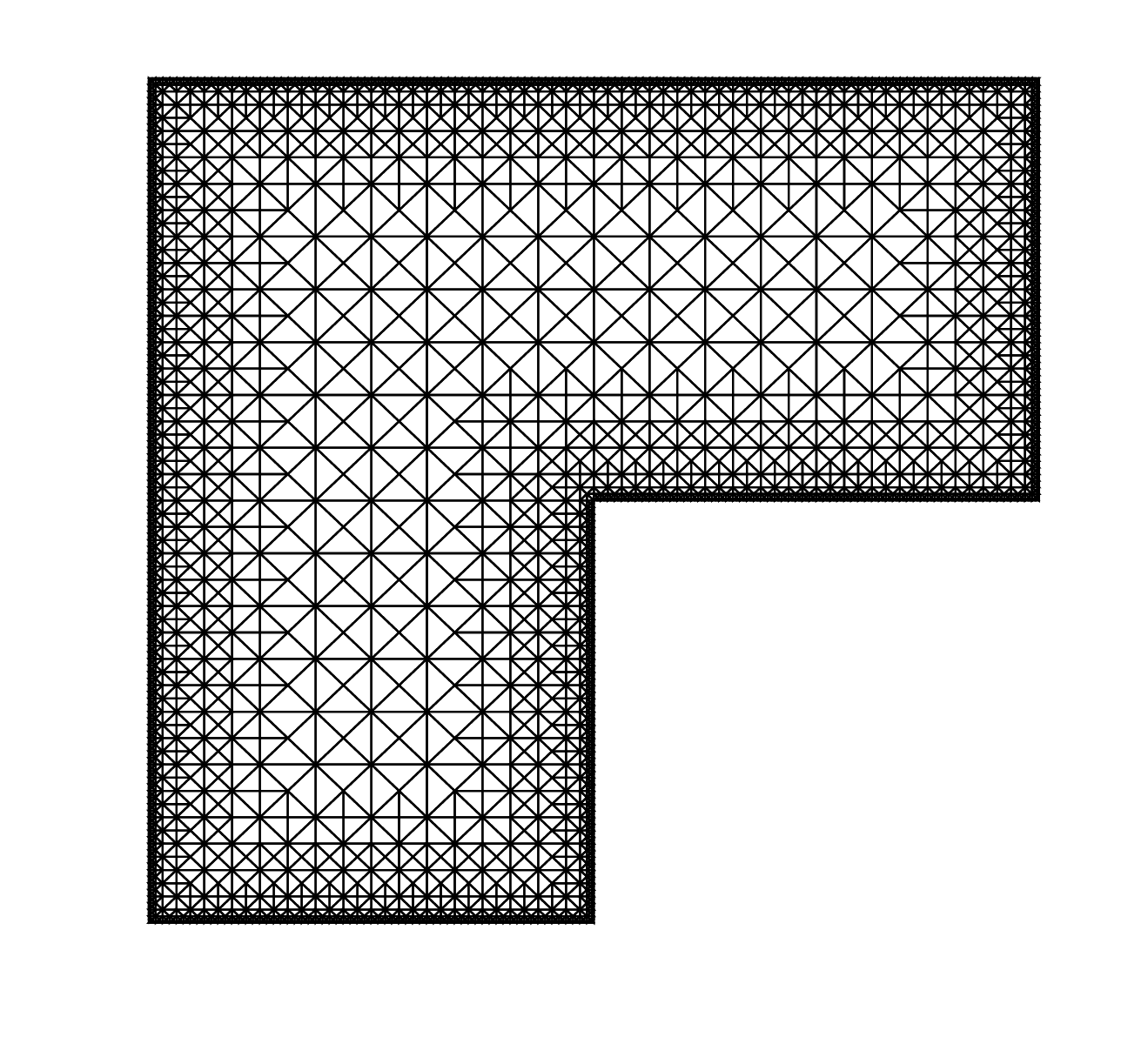} \hspace{-.5cm}
 \includegraphics[width=0.34\linewidth]{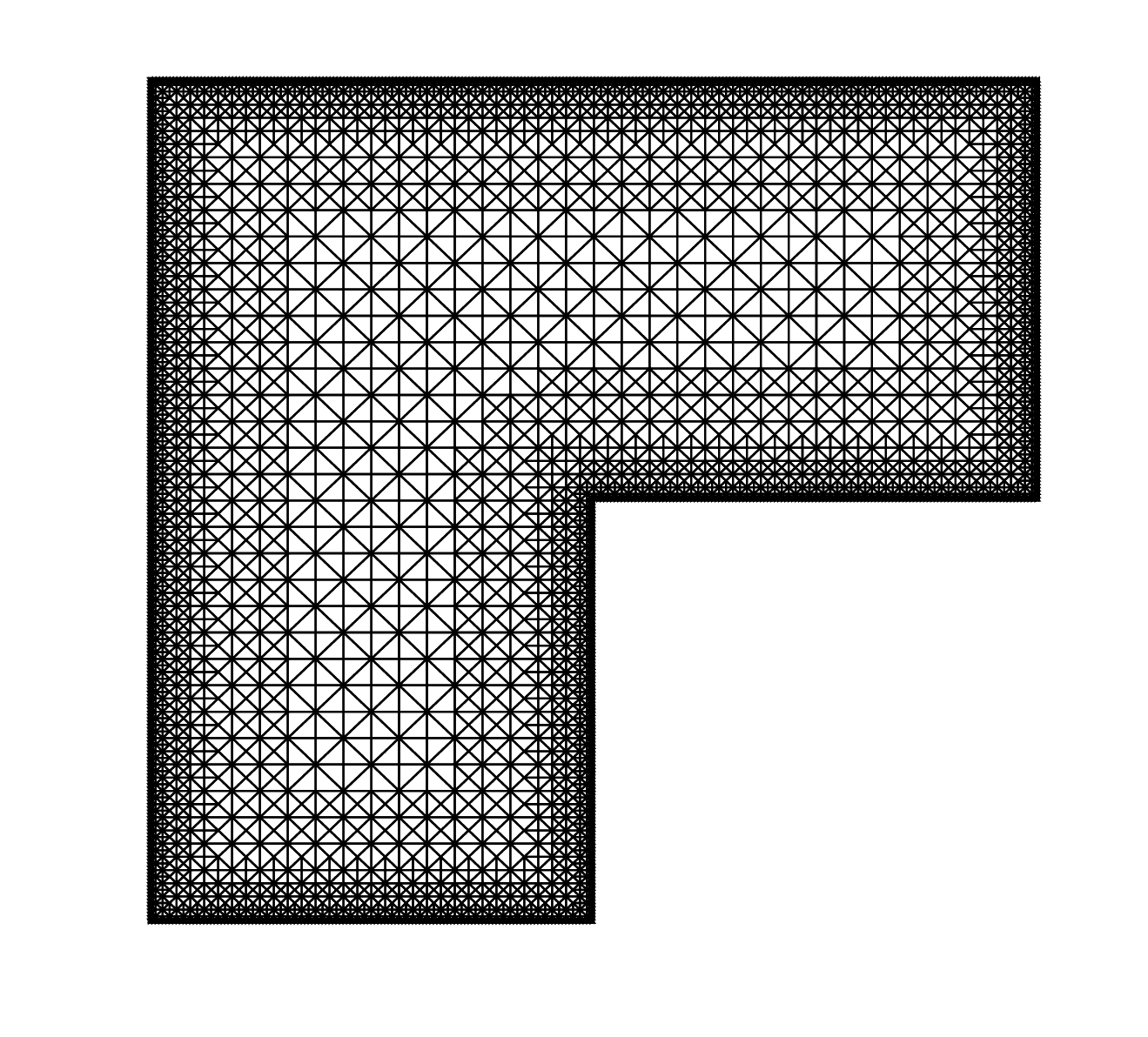}\hspace{-.3cm}
 \includegraphics[width=0.34\linewidth]{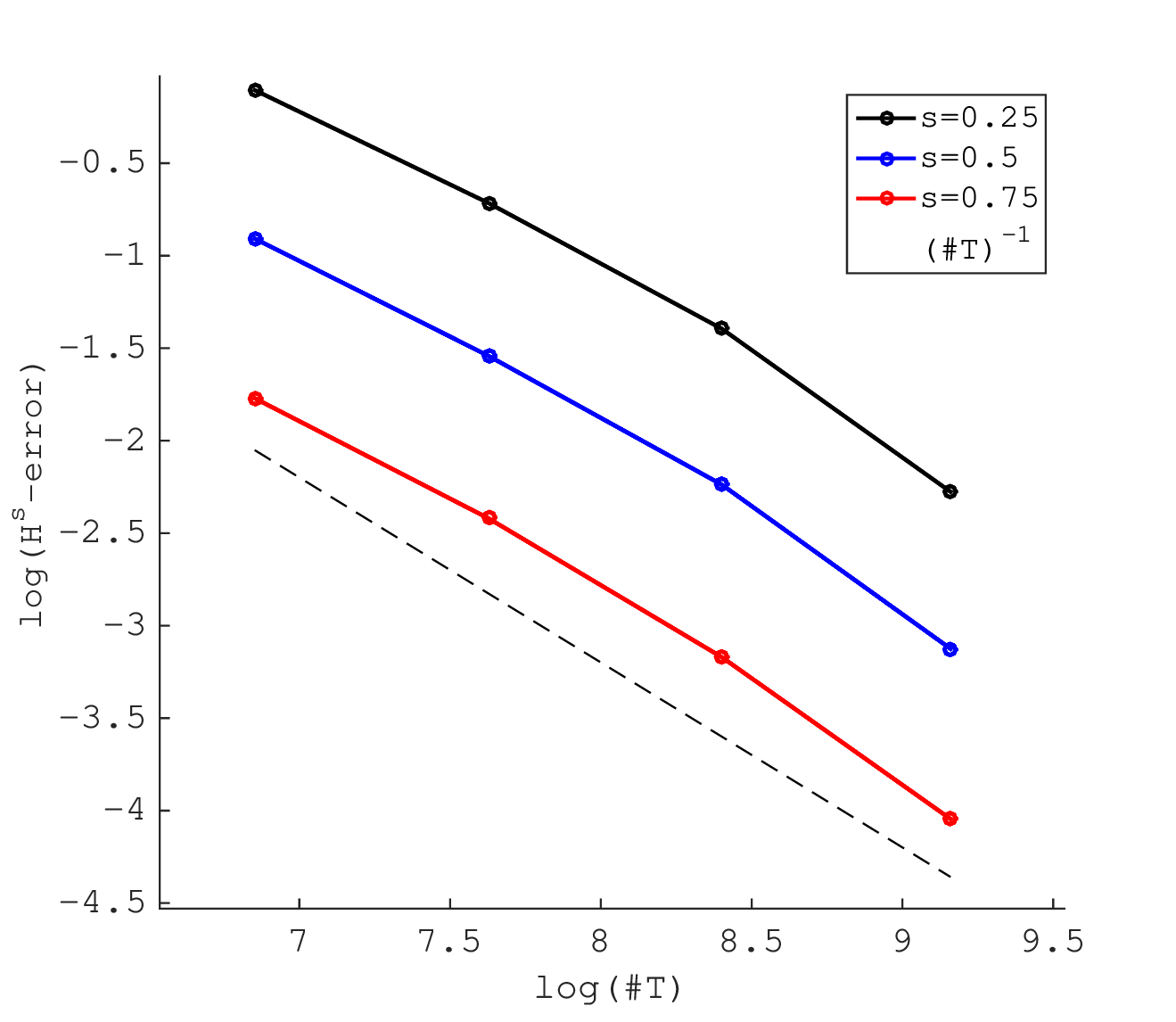}  
\end{center}
\vspace{-0.5cm}
\caption{$\GREEDY$ with surrogate estimator \eqref{eq:estimator}. Left and center: graded bisection meshes with 9504 and 15118 elements, respectively. Right: errors in the $\tHs$-norm for $s \in \{0.25, 0.5, 0.75\}$ and $f = 1$. Computational rates for the $L$-shaped domain are consistent with the expected theoretical rate $\big(\#\T\big)^{-\frac12}|\log \#\T|^2$ for solutions $u \in \wt{W}^{s+1-\eps}_{1+\eps}(\Omega)$ (Theorem \ref{T:greedy}).}
\label{fig:numerical}
\end{figure}
To create a sequence of meshes $\T_k$ and examine the error decay $\|u-\Pi_{\T_k}u\|_{\wt{H}^s(\Omega)}$ in terms of $\#\T_k$, we run $\GREEDY$ with tolerance $\delta_k = 2^{-k} \cdot 10^{-2}$, $k = 2, \ldots 5$. We stress that the marking strategy $\mathcal{E}_T(u) > \delta_k$ is independent of $s$ and insensitive to the presence of reentrant corners. This is reflected in Figure \ref{fig:numerical}, whose left and middle panels depict meshes with $\#\T = 9504$ elements and $\#\T = 15118$ elements, respectively.
Due to the lack of a closed expression for the solution $u$ of \eqref{eq:Dirichlet} in this setting, we used a solution on a highly refined mesh as a surrogate of $u$ to compute the desired error $\|u-\Pi_{\T_k}u\|_{\wt{H}^s(\Omega)}$. The right panel in Figure \ref{fig:numerical} displays our computational orders of convergence. They show a good agreement with the expected log-linear rate from Theorem \ref{T:greedy} (quasi-optimal error estimate for $d \ge 2$), even though $\Omega$ does not satisfy the sufficient conditions leading to \eqref{eq:Holder-regularity}. We refer to Remark \ref{rmk:corners} (corner singularities in two dimensions) for further discussion.


\bibliography{adaptive}
\bibliographystyle{abbrv}
\end{document}